\documentclass[11pt,oneside,reqno]{amsart}
\usepackage{amsthm}
\usepackage{amsrefs}
\usepackage{amssymb}
\usepackage{amsmath}
\usepackage{graphicx}									 %Include graphics
\usepackage{stmaryrd}                  %Includes \mapsfrom
\usepackage{footmisc}                  %Something for footnotes
\usepackage{paralist}
\usepackage{wrapfig}
\usepackage[draft]{pdfcomment}
\usepackage{bbm}
\usepackage[T2A,T1]{fontenc}
\usepackage[utf8]{inputenc}
\usepackage{lmodern}
\usepackage[ngerman,russian,british]{babel}
\usepackage[arrow, matrix, curve]{xy}
\usepackage{xcolor}
\usepackage{capt-of}
\usepackage{enumitem}
\usepackage[utf8]{inputenc}
\usepackage{float}
\usepackage{todonotes}
\usepackage{hyperref}
\usepackage{aliascnt}
\usepackage{enumitem}
\usepackage{tikz-3dplot}
\usetikzlibrary{shapes,backgrounds}
\renewcommand{\phi}{\varphi}
\title[Alexandrov Regions]{Alexandrov Regions}

\author[N. Lebedeva]{Nina Lebedeva\textsuperscript{1}}
\author[A. Nepechiy]{Artem Nepechiy\textsuperscript{2}}

\thanks{%
  \textsuperscript{1}\,Saint Petersburg State University,
  7/9 Universitetskaya nab., St. Petersburg, 199034, Russia;
  St. Petersburg Department of V.~A.~Steklov Institute of Mathematics
  of the Russian Academy of Sciences,
  27 Fontanka nab., St. Petersburg, 191023, Russia.
  Email:
  \href{mailto:lebed@pdmi.ras.ru}{\nolinkurl{lebed@pdmi.ras.ru}}.
}

\thanks{%
  \textsuperscript{2}\,Institute for Algebra and Geometry,
  KIT, Englerstr.~2, 76131 Karlsruhe, Germany.
  Email:
  \href{mailto:artem.nepechiy@kit.edu}{\nolinkurl{artem.nepechiy@kit.edu}}.
}

\makeatletter

\@namedef{subjclassname@2020}{%
  \textup{2020} Mathematics Subject Classification}

\let\original@adminfootnotes\@adminfootnotes
\renewcommand{\@adminfootnotes}{%
  \begingroup
  \long\def\@makefntext##1{%
    \setlength{\parindent}{0pt}%
    \noindent\@makefnmark##1%
  }%
  \original@adminfootnotes
  \endgroup
}

\makeatother



\newtheorem{rem}{Remark}[section]
\newtheorem{defi}{Definition}[section]

\DeclareMathOperator{\dist}{dist}
\DeclareMathOperator{\cone}{Cone}

\newtheorem*{Main Theorem}{Convex Neighborhood Theorem}

\newtheorem*{Quantitative Main Theorem}{Quantitative Convex Neighborhood Theorem}

\newaliascnt{thm}{defi}
\newtheorem{thm}[thm]{Theorem}

\newtheorem{prop}[thm]{Proposition}

\aliascntresetthe{thm}

\newaliascnt{prop}{defi}
\aliascntresetthe{prop}

\newaliascnt{cor}{defi}
\newtheorem{cor}[cor]{Corollary}
\aliascntresetthe{cor}

\newaliascnt{lem}{defi}
\newtheorem{lem}[lem]{Lemma}
\aliascntresetthe{lem}

\newaliascnt{ex}{defi}

\aliascntresetthe{ex}

\newcommand{\N}{\mathbb{N}}

\newcommand{\R}{\mathbb{R}}
\begin{document}

\begin{abstract}
	We introduce the notion of Alexandrov region. In contrast to Alexandrov spaces these are intrinsic metric spaces, which satisfy a lower curvature bound and have finite Hausdorff dimension, might however, be non-complete. We show a quantitative version of the statement that such regions are locally isometric to an Alexandrov space.
\end{abstract}
	\maketitle

\section{Introduction}
The local lower curvature condition goes back to Alexandrov
\cite{MR87119} and appears in its modern metric form in
\cite{MR1185284}. We use the term \emph{Alexandrov region} for
a locally compact, finite-dimensional intrinsic metric space satisfying
this local comparison condition; see \autoref{Def: Alexandrov region}. Basic examples are
convex open subsets of finite-dimensional $CBB(\kappa)$ spaces. A useful
example beyond the complete setting is the universal cover of
$\mathbb R^2\setminus \lbrace {0} \rbrace $ with the lifted metric; its metric completion
is not locally compact.

Our first result makes precise the fact that the local comparison
condition already provides genuine Alexandrov neighborhoods.

\begin{prop}[Alexandrov region is locally Alexandrov]\label{main theorem}
Let $X^n$ be an Alexandrov region of dimension $n$. Then for every
$p\in X$ there exists $p\in A\subset X$ such that $A$ with the restricted metric is an $n$-dimensional Alexandrov space.
\end{prop}

In fact, the neighborhoods in Proposition~\ref{main theorem} can be
chosen arbitrarily small. This is essentially a local consequence of
Perelman's construction of strictly concave functions. We include the
details mainly to justify its use in the non-complete setting and to
verify that the construction can be lifted along Gromov--Hausdorff
convergent sequences.

Proposition~\ref{main theorem} alone gives no uniform control on the
size of such a neighborhood: along a sequence of Alexandrov regions
the available convex neighborhoods may shrink to zero. The main point
of this paper is that this degeneration can be ruled out quantitatively.
More precisely, under a non-collapsing assumption, the size of an
Alexandrov neighborhood admits a uniform lower bound in terms of the
distance to the virtual boundary (points of the metric completion that do not belong to the original space).

The following formulation uses the compactness of $\overline B_r(p)$ as a local substitute for a controlled distance from the virtual boundary, while the volume lower bound prevents collapse.

\begin{thm}[Quantitative convex neighborhood theorem]\label{quantitative theorem}
For every $\varepsilon>0$, $n\in\mathbb N$, and $\kappa\in\mathbb R$ there exists $c=c(\varepsilon,n,\kappa)>0$ with the following property. Let $X^n$ be an Alexandrov region with curvature bounded below by $\kappa$, and let $p\in X$. Suppose that $0<r\le1$, $\overline B_r(p)$ is compact, and

$$
\mathcal H^n(B_{r/2}(p))\ge \varepsilon r^n.
$$

Then there exists an $n$-dimensional Alexandrov space $A$ such that

$$
B_{cr}(p)\subset A\subset B_r(p).
$$

\end{thm}

The underlying quantitative principle has already been used implicitly, for instance in \cites{unknown,alex2018ricci}; however, we are not aware of an explicit proof in the present generality.

The proof combines a quantitative version of the globalization argument with a compactness argument. The comparison radius is first controlled by iterating a radius-improvement step, with the remaining distance to the boundary serving as a reserve. This gives uniform local packing estimates and hence Gromov--Hausdorff precompactness on a smaller ball. The non-collapsing assumption prevents a drop of dimension in the limit. Perelman's strictly concave function can then be constructed on the limit and lifted back to the approximating spaces, producing convex neighborhoods of uniformly positive size.

This local framework is useful precisely in situations where Alexandrov geometry appears before one has a complete global space. Gromov--Hausdorff limits occur naturally in degeneration and collapsing problems, in rescaling and blow-up arguments, and more generally throughout geometric analysis. Under lower sectional curvature bounds the limiting objects belong to Alexandrov geometry, but the local pieces and intermediate spaces to which one wishes to apply comparison arguments need not themselves be complete. The quantitative theorem provides a uniform scale on which such pieces can nevertheless be replaced by genuine Alexandrov spaces.

Alexandrov regions also arise naturally in more specific limiting and quotient constructions. In particular, they occur in the study of limits of locally homogeneous spaces \cites{MR3984075,pediconi2019compactness} and in constructions related to orbit spaces and singular Riemannian foliations \cite{alex2018ricci}. Similar local questions occur for metric quotients and submetries, where lower curvature bounds are naturally preserved but one often works on open pieces or strata rather than on the whole quotient. Thus the results above allow the usual tools of Alexandrov geometry---comparison arguments, spaces of directions and tangent cones, strainer techniques, and stability methods---to be applied on a controlled neighborhood without first imposing global completeness.

\subsection*{Use of AI assistance} In the preparation of this manuscript, the authors used ChatGPT (OpenAI, GPT-6) for proofreading, correcting typos, and improving linguistic style. The authors take full responsibility for the content of the paper.

\subsection*{Acknowledgements} The second named author would like to thank the Hausdorff Center for Mathematics in Bonn for their support and hospitality. Part of this work was completed during the second author's stay at the Trimester program \emph{Geometric Statistics: Theory, Application, and Computation}.
\section{Alexandrov Region}

	In this section we introduce the notion of an Alexandrov region and establish the local Alexandrov geometry needed below. Most of the constructions are standard for Alexandrov spaces, but their usual proofs are formulated in the complete setting. We therefore check that the relevant arguments are local and can be carried out inside a compact working ball, without using globalization or compactness of the whole space. In particular, we obtain that the space of directions and the tangent space are well defined at every point and are Alexandrov spaces; see \autoref{Cor: Tangent space and space of directions are Alexandrov spaces}.
We assume familiarity with the exposition in \cite{MR1835418}, but include the relevant details for this reason.

There are several definitions of a lower curvature bound. Some of them do not require the space to be geodesic.

	\begin{defi}[Lower curvature bound]\label{Def: Lower curvature bound}
		A metric space $(X,d)$ has a lower curvature bound $\geq \kappa$ if every ${p} \in X$ has a {
		neighborhood $U_p$ containing $p$} such that for all $q;a,b,c \in U_p$ one has

\begin{equation}\label{eq: comparison inequality}
%\label{4pt}
		 \sphericalangle_{\kappa} a{q}b + \sphericalangle_{\kappa} b{q}c + \sphericalangle_{\kappa} c{q}a \leq 2 \pi. 
	\end{equation}

	\end{defi}

We now come to the definition of Alexandrov regions. Compared to an Alexandrov space it is missing the completeness condition. Since we want to exclude the infinite-dimensional case we have to build it into the definition. 
	\begin{defi}[Alexandrov region]\label{Def: Alexandrov region}
		A locally compact, connected, intrinsic metric space $(X,d)$ of finite Hausdorff dimension is called an Alexandrov region if it has a lower curvature bound in the sense of \autoref{Def: Lower curvature bound}.
	\end{defi}

Since completeness and global geodesicity are not assumed, we work inside controlled local neighborhoods. 

\begin{defi}[Working radius]
We say that $r>0$ is a \emph{working radius} at $p\in X$ if
$\overline{B_{6r}(p)}$ is compact and the comparison inequality
\eqref{eq: comparison inequality} holds for every quadruple of points in
$\overline{B_{4r}(p)}$. \par 
A ball $B_r(p)$, where $r$ is a working radius, is called a working ball.
\end{defi}

Every point admits a positive working radius. Moreover, if $r$ is a working
radius at $p$, then any two points in $\overline{B_{2r}(p)}$ can be joined by
a minimizing geodesic contained in $\overline{B_{4r}(p)}$.
Thus, inside a working ball we have both ingredients needed for the
usual local arguments: minimizing geodesics and the comparison
inequality.
We first record a dimensional consequence of this local structure. It
will be used below to identify the dimension of tangent cones.

\begin{lem}[Dimension homogeneity] Let $X$ be an Alexandrov region. Then all nonempty open subsets have the same Hausdorff dimension. \end{lem} \begin{proof}
The argument is the local version of
\cite[Theorem~10.6.1 and Lemma~10.6.2]{MR1835418};
compare also \cite[Exercise~10.6.3]{MR1835418}.
Since only sufficiently short minimizing geodesics are used, the same
radial contraction argument applies inside a working neighborhood.

Since $X$ is connected and locally compact, it is separable. Hence it can be covered by countably many  such neighborhoods, which  shows that their common dimension equals $\dim_H X$, and therefore every nonempty open subset has dimension $\dim_H X$. \end{proof}

We now turn to the infinitesimal geometry of an Alexandrov region.
Since the definition of angles uses only sufficiently short geodesics
and local comparison, the usual construction can be carried out inside
a working ball.  
	Let $X$ be an Alexandrov region, $p \in X$ and $\gamma_1,\gamma_2: [0,\varepsilon)\rightarrow X$ two shortest paths with $\gamma_1(0)=p=\gamma_2(0)$. We can define their angle as 
        $$ \sphericalangle(\gamma_1, \gamma_2) := \lim_{s,t \rightarrow 0} \sphericalangle_{\kappa} \gamma_1(s)p \gamma_2(t).$$
    In an Alexandrov region the limit above always exists, 
    and $\sphericalangle(\gamma_1,\gamma_2)$ satisfies the triangle inequality \cite{MR1835418}[Theorem 3.6.34]. Let $r(p)>0$ be a working radius. Consider the set $S_p$ given by all shortest curves starting at $p$ and ending at some $x \in B_{r/4}(p)$ together with the equivalence relation $\gamma_1 \sim \gamma_2$ whenever $\sphericalangle(\gamma_1,\gamma_2)=0$. This makes $(S_p / \sim, \sphericalangle)$ into a metric space. We can now define the space of directions.
	\begin{defi}[Space of directions]
		Let $p \in X$ be a point in an Alexandrov region. Consider the metric space $(S_p / \sim, \sphericalangle)$ defined as above. Then the space of directions $\Sigma_p(X)$ is defined to be the metric completion of $(S_p / \sim, \sphericalangle)$.
	\end{defi}
Compactness of $\Sigma_p$ will follow from local packing estimates obtained via bi-Lipschitz distance coordinates associated with strainers (\autoref{Theorem: Bilipschitz Homeo around strainer}).

	\begin{defi}[Strainer]
		Let $X$ be an Alexandrov region. A point $p \in X$ is an $(m,\varepsilon)$-strained point if there are
		$m$ pairs of points $(a_i, b_i)$ in a working ball around $p$ such that
			$$\sphericalangle_{\kappa}a_i p b_i \geq \pi - \varepsilon, \quad \sphericalangle_{\kappa}a_i p a_j \geq \frac{\pi}{2} - \varepsilon,\quad \sphericalangle_{\kappa}b_i p b_j \geq \frac{\pi}{2} - \varepsilon,\quad \sphericalangle_{\kappa}a_i p b_j \geq \frac{\pi}{2} - \varepsilon.  $$
		for all $i, j \in \lbrace 1, \ldots, m \rbrace $ with $i \neq j$. The collection $\lbrace(a_i, b_i)\rbrace$ itself is called an
		$(m,\varepsilon)$-strainer for $p$.
	\end{defi}
	\begin{defi}[Strainer Number]
		Let $X$ be an Alexandrov region. For $n \in \mathbb{N}$ and $\varepsilon>0$ denote by $\operatorname{Str}^n(X,\varepsilon)$ the set of $(n,\varepsilon)$-strained points in $X$. The strainer number
			$$\operatorname{Strain}(X,\varepsilon):= \sup  \lbrace m \in \mathbb{N}: \operatorname{Str}^m(X,\varepsilon)  \neq \emptyset \rbrace$$
		is the supremum of numbers $m$ such that there exists an $(m,\varepsilon)$-strainer
		in $X$.	The local strainer number  at a point $p \in X$ is
			$$\operatorname{Strain}(X,\varepsilon,p):= \sup \lbrace m \in \mathbb{N}: \operatorname{Str}^m(X,\varepsilon) \cap B_r(p)  \neq \emptyset \text{ for all } r>0 \rbrace  $$
		 the supremum of numbers $m$ such that every neighborhood of $p$ contains an $(m,\varepsilon)$-strained point.
	\end{defi}
As for Alexandrov spaces we have the following key result.

\begin{thm}[Distance coordinates near a strained point]\label{Theorem: Bilipschitz Homeo around strainer}
Let $p\in X$ be an $(m,\varepsilon)$-strained point in an Alexandrov region, where
$\varepsilon>0$ is sufficiently small depending only on $m$.

\begin{enumerate}
\item The associated distance map is open near $p$. In particular,
\[
m\le \dim_H X.
\]

\item If $m=\operatorname{Strain}(X,\varepsilon,p)$, then the distance
map is bi-Lipschitz on a neighborhood of $p$.
\end{enumerate}
\end{thm}
\begin{proof}
The argument is local. By
\cite[Proposition~10.8.15]{MR1835418}, the distance map associated
with the strainer is open in a neighborhood of $p$. Since this map is
Lipschitz and its image contains an open subset of $\mathbb R^m$, we get
$m\le \dim_H X;$
compare \cite[Corollary~10.8.16]{MR1835418}.

If, in addition,
$
m=\operatorname{Strain}(X,\varepsilon,p),
$
the arguments of
\cite[Propositions~10.8.12 and~10.8.15,
Lemmas~10.8.13 and~10.8.14,
Theorem~10.8.18]{MR1835418}
show that, after shrinking the neighborhood of $p$, the same distance
map is bi-Lipschitz.

All points and minimizing geodesics involved can be chosen inside a
working neighborhood of $p$, so the proof applies without a global
completeness assumption.
\end{proof}

We  use this local Euclidean parametrization to bound separated
sets of directions.

	\begin{prop}[Space of directions is compact]\label{Prop: Space of directions is compact}
		Let $X$ be an Alexandrov region and $p \in X$ a point. Then the space of directions $\Sigma_p(X)$ is compact.
	\end{prop}
\begin{proof}
By construction, $\Sigma_p(X)$ is complete, so it remains to prove
total boundedness. We repeat the proof of
\cite[Proposition~10.9.1]{MR1835418}.
The packing estimate used there,
\cite[Lemma~10.9.2]{MR1835418}, is local. Its proof uses a
bi-Lipschitz distance chart near a point of maximal strainer rank
and the radial contraction argument of
\cite[Lemma~10.6.2]{MR1835418}.
The former is available here by Theorem~2.8, while the latter can be
carried out inside a working neighborhood of $p$.

Consequently, for every $\varepsilon>0$ there is a uniform finite
upper bound on the cardinality of an $\varepsilon$-separated subset
of $\Sigma_p(X)$, exactly as in the proof of
\cite[Proposition~10.9.1]{MR1835418}.
Thus $\Sigma_p(X)$ is totally bounded, and hence compact.
\end{proof}

Compactness of $\Sigma_p$ allows us to carry out the standard blow-up
argument. Observe that for fixed $R>0$ the ball $B_R(p) \subset \frac{1}{r}X$ is compact for sufficiently small $r$, hence the notion of pointed GH-convergence is well-defined.\par 
Since every fixed bounded ball in the rescaled space
$r^{-1}X$ lies, for sufficiently small $r$, inside a working
neighborhood of $p$, only the local geometry of $X$ enters the
argument.

\begin{lem}[Convergence to tangent space]\label{Lemma: Convergence to tangent space}
			Let $X$ be an Alexandrov region and $p \in X$ an arbitrary point. Then one has for $r \rightarrow 0$
				$$\left(\frac{1}{r}X,p \right) \xrightarrow{d_{p.GH}} \left(\operatorname{Cone}(\Sigma_p(X)),o_p \right),$$
			where $\operatorname{Cone}(Y)$ denotes the metric cone over a metric space $Y$ and $d_{p.GH}$ denotes the pointed Gromov--Hausdorff convergence.
	\end{lem}
	\begin{proof}		
        Having \autoref{Prop: Space of directions is compact}, we apply the
standard proof of \cite[Theorem 10.9.3, pp.~391--392]{MR1835418}.
For every fixed bounded ball in the rescaled spaces, all points and
minimizing geodesics involved lie, for sufficiently small $r$, inside
a working ball at $p$. Thus the proof uses only local geodesicity,
local comparison, and compactness of $\Sigma_p$.
	\end{proof}
It remains to verify that the limiting cone is a genuine Alexandrov
space and to identify its dimension.

	\begin{cor}[Tangent space and space of directions are Alexandrov spaces]\label{Cor: Tangent space and space of directions are Alexandrov spaces}
		Let $X$ be an $n$-dimensional Alexandrov region and $p \in X$ an arbitrary point. Then $\operatorname{Cone}(\Sigma_p(X)) $ is an
        $n$-dimensional Alexandrov space with curvature $\geq 0$. Thus $\Sigma_p(X)$ is an
        $(n-1)$-dimensional Alexandrov space with curvature $\geq 1$ or consists of two points.
	\end{cor}
		\begin{proof}

        Since $\Sigma_p(X)$ is compact by \autoref{Prop: Space of directions is compact}, $\operatorname{Cone}(\Sigma_p(X))$
is proper. Moreover, on every fixed bounded ball, the rescaled spaces
$r^{-1}X$ are geodesic for all sufficiently small $r$: all points and
minimizing geodesics involved lie inside a working neighborhood of $p$.
Hence the standard stability of geodesic spaces under
Gromov--Hausdorff convergence implies that
$\operatorname{Cone}(\Sigma_p(X))$ is geodesic.

         Now $\operatorname{Cone}(\Sigma_p(X))$ 
            is the GH-limit of spaces 
            where \eqref{eq: comparison inequality} holds
            for every quadruple of points in a small neighborhood of $p$
            with lower curvature bound converging to zero, hence 
            it is an Alexandrov space
            with curvature $\geq 0$.\par
			The dimension of $\operatorname{Cone}(\Sigma_p(X))$ cannot exceed $n$. Indeed, GH approximations carry $(m,\varepsilon)$-strainers around $o_p$ to $(m,\varepsilon)$-strainers around $p$ in $r^{-1}X$. This implies $$\operatorname{Strain}(\operatorname{Cone}(\Sigma_p),\varepsilon,o_p) \leq \operatorname{Strain}(r^{-1}X,\varepsilon,p) .$$ By  Theorem~\ref{Theorem: Bilipschitz Homeo around strainer}, the local strainer number is bounded above by the Hausdorff dimension of $X$. 
            The opposite inequality for dimension follows since $\log^\kappa_p:B_r(p)\to \operatorname{Cone}_\kappa(\Sigma_p(X))$ is a noncontracting map.
            
            Therefore $\operatorname{Cone}(\Sigma_p(X))$ is an $n$-dimensional Alexandrov space. The last assertion now follows from \cite{MR1835418}[Theorem 10.2.3, p. 355].
		\end{proof}

\section{Strictly concave functions in Alexandrov regions}
To construct arbitrarily small convex neighborhoods, we need a local source of strict concavity. Indeed, taking convex hulls is not a priori a local operation: the convex hull of a small set may leave the prescribed neighborhood. Perelman's construction provides strictly concave functions locally; moreover, it can be arranged so that the resulting function has a strict local maximum at a prescribed point. Suitable superlevel sets then give arbitrarily small convex neighborhoods.

The construction is standard in Alexandrov geometry; see \cites{MR1220498,MR2408266, MR2408265, nepechiy2019canonical, zbMATH07802912,CBBlec}. We give the details needed here for two reasons. First, we have to verify that all arguments are local and remain valid for Alexandrov regions. Second, the quantitative result will require the construction to be lifted along non-collapsing Gromov--Hausdorff convergent sequences. We therefore keep track of the finite geometric data and the estimates entering the construction.

\subsection{
Preliminary notions}
We first recall the notions of semiconcavity and differential used below.

\begin{defi}[$\lambda$-concavity]\label{Def: lambda concavity}
Let $X$ be an Alexandrov region and $\Omega\subset X$ be an open working neighborhood. A locally Lipschitz function
$f:\Omega\to\mathbb R$ is called \emph{$\lambda$-concave} if, for every unit-speed geodesic $\gamma$ contained in $\Omega$,
$$
f\circ\gamma(t)-\frac{\lambda}{2}t^2
$$
is concave. The function $f$ is called \emph{semiconcave} if every point has a neighborhood on which $f$ is $\lambda$-concave for some $\lambda$.
\end{defi}

As in the usual Alexandrov setting, the distance function
$$
\dist_q(x)=|qx|
$$
is locally semiconcave away from $q$; the standard proof is local and applies inside a working ball.

For a semiconcave function $f$ and a geodesic $\gamma$ starting at
$p$, the one-sided derivative
$$
(f\circ\gamma)^+(0)
=
\lim_{t\searrow0}\frac{f(\gamma(t))-f(p)}{t}
$$
exists. If $\gamma$ is a minimizing geodesic from $p$ to $q$, we denote
this derivative by $f'_p(\uparrow_p^q)$. It depends only on the direction
$\uparrow_p^q\in S_p(X)$ and extends uniquely to $\Sigma_p(X)$. Extending
it radially gives
$$
d_pf:T_pX\to\mathbb R,
$$
called the differential of $f$ at $p$. For $v\in T_pX$, define its \emph{absolute value} by
\(
|v|:=d_{T_pX}(o_p,v).
\)
For $u,v\in T_pX$, define their \emph{scalar product} by
\[
\langle u,v\rangle
:=
\frac{|u|^2+|v|^2-d_{T_pX}(u,v)^2}{2}.
\]

We now formulate Perelman's local strict concavity construction
for Alexandrov regions.
\begin{prop}[Existence of strictly concave functions]\label{Prop:Existence of strictly concave functions}
Let $A$ be an $n$-dimensional Alexandrov region, and let $p\in A$. Then there exists a strictly concave function $f$ in a small neighborhood of $p$.
   Moreover, for any vector $v\in T_pA$ 
   the differential $d_{p}f(\cdot)$ can be chosen arbitrarily close to the map 
   $ - \langle v,\cdot \rangle$.

\end{prop}

\subsection{Plan of the proof}

The main analytic ingredient is a simple one-dimensional observation for improving concavity. Let $u$ be a $\lambda$-concave function and let $\phi$ be smooth, increasing, with $\phi'\approx1$ and $\phi''<-M$. Formally,

$$
(\phi\circ u)''=\phi'(u)u''+\phi''(u)(u')^2.
$$

Thus, if $|u'|$ is bounded away from zero, the second term makes $\phi\circ u$ strictly concave for sufficiently large $M$.

We apply this to $u=\dist_q\circ\gamma$, where $\gamma$ is a unit-speed geodesic. The first variation formula gives

$$
(\dist_q\circ\gamma)'(t)=-\cos\alpha(t),
$$

where $\alpha(t)$ is the angle between $\gamma$ and the direction toward $q$.  Hence, in the smooth case

$$
(\phi\circ\dist_q\circ\gamma)''(t)
=
\phi'(\dist_q)(\dist_q\circ\gamma)''
+\phi''(\dist_q)\cos^2\alpha(t).
$$

The first term is bounded above by semiconcavity of the distance function, while the second is uniformly negative whenever $\alpha$ stays away from $\pi/2$. Thus composition with $\phi$ improves concavity in every direction which is not almost orthogonal to the direction toward $q$.

The packing argument (Lemma~\ref{pi/2}) provides finitely many points $q_1,\ldots,q_N$
and a constant $\xi>0$ such that every geodesic direction makes an
angle at least $\xi$ away from $\pi/2$ with the direction to some
$q_i$. Consequently, for that $i$,
$
\cos^2\alpha_i\ge \sin^2\xi=:c_0>0.
$

 Averaging the functions $\phi\circ\dist_{q_i}$ then gives

$$
\frac1N\sum_{i=1}^N \phi\circ\dist_{q_i},
$$

whose second derivative is bounded above by

$$
2\lambda-\frac{Mc_0}{N},
$$

and is therefore negative for sufficiently large $M$.

The nonsmooth proof follows exactly this scheme. Lemma \ref{deriv2} replaces the formal chain-rule computation by a barrier estimate, while Lemma~\ref{pi/2} supplies the required finite family of nonorthogonal directions.

\subsection{Barrier Second Derivatives }
We now formulate the preceding one-dimensional argument in the nonsmooth setting.

A smooth function $\bar u$, defined near $t_*\in I\subset\R$, is an \emph{upper barrier} for $u$ at $t_*$ if
$\bar u(t_*)=u(t_*)$
and $\bar u\ge u$ in a neighborhood of $t_*$. We write
$
u''(t_*)\le c
$
in the barrier sense if there exists such an upper barrier with
$$
\bar u''(t_*)\le c.
$$

We use the following standard one-dimensional characterization.

\begin{lem}[$\lambda$-concavity in the barrier sense]\label{lem}
A locally Lipschitz function $u:I\to\mathbb R$ is $\lambda$-concave if and only if
$$
u''(t_*)\le\lambda
$$
in the barrier sense at every interior point $t_*\in I$.
\end{lem}

See \cite[Theorem~3.14]{zbMATH07802912}.
 
The following lemma is the nonsmooth version of the chain-rule estimate above.

\begin{lem}[Barrier chain-rule estimate]\label{deriv2}
Let $u:(a,b)\to\R$ be a $\lambda$-concave function
with $\lambda>0$.
Let
$t^*\in(a,b), M>0$, and let $\phi:\R\to\R$ be smooth such that
$$
0<\phi'(u(t^*))<2,
\qquad
\phi''(u(t^*))\le -2M.
$$
Then
$$
(\phi\circ u)''(t^*)
<
2\lambda
-
M\Bigl(\max\{|u'_-(t^*)|,|u'_+(t^*)|\}\Bigr)^2
$$
in the barrier sense at $t^*$.
\end{lem}

\begin{proof} Denote by $
u'_+(t^*)\le u'_-(t^*)
$ the one-sided derivatives of $u$. 
Since $u$ is $\lambda$-concave, one has
$
u'_+(t^*)\le u'_-(t^*).
$
 If $u'_+(t^*)<u'_-(t^*)$, then, since $\phi'>0$,
$\phi\circ u$ has a strict downward corner at $t^*$. Choosing a slope
strictly between its one-sided derivatives shows that it admits upper
barriers with arbitrarily negative second derivative.

Assume now that
$
u'_+(t^*)=u'_-(t^*)=:u'(t^*).
$
Choose a smooth upper barrier $u_*\ge u$ at $t^*$ such that
$
u_*''(t^*)\le\lambda.
$
Since $\phi'(u(t^*))>0$, the function $\phi$ is increasing near
$u(t^*)$, and therefore $\phi\circ u_*$ is an upper barrier for
$\phi\circ u$. Moreover,
$
u_*'(t^*)=u'(t^*).
$
Thus
$$
\begin{aligned}
(\phi\circ u_*)''(t^*)
&=
\phi'(u(t^*))u_*''(t^*)
+
\phi''(u(t^*))(u_*'(t^*))^2\\
&<
2\lambda-M(u'(t^*))^2,
\end{aligned}
$$
which proves the claim.
\end{proof}

\iffalse

\fi
\subsection{A finite family of uniformly nonorthogonal directions}

The geometric ingredient is the following uniform nonorthogonality
statement.

\begin{lem}[Finite nonorthogonal family]\label{pi/2}
Let $X^n$ be an Alexandrov region and $p\in X$.
For every sufficiently small $r>0$ there exist
$\xi>0$, $\varepsilon>0$, and points
$q_1,\dots,q_N\in S_r(p)$ with the following property.
For every $x\in B_\varepsilon(p)$ and every direction
$v_x\in\Sigma_x$, there exist $i\in\{1,\dots,N\}$ and a
minimizing geodesic from $x$ to $q_i$ such that
$
\left|
\frac{\pi}{2}-\sphericalangle(v_x,\uparrow_x^{q_i})
\right|>\xi.
$

The points $q_i$ may be chosen so that the minimizing geodesics
$pq_i$ are unique. Moreover, given $v_p\in\Sigma_p$ and $\nu>0$,
they may be chosen so that
$
\sphericalangle(\uparrow_p^{q_i},v_p)<\nu
\quad\text{for all }i.
$
\end{lem}

The proof is a dimension-counting argument. A small ball in the
$(n-1)$-dimensional space $\Sigma_p$ contains on the order of
$\delta^{-(n-1)}$ separated directions, whereas a
$\delta$-neighborhood of the $\pi/2$-level set of the distance
from a fixed direction contains at most on the order of
$\delta^{-(n-2)}$ such directions.

For a metric space $Y$ and $\delta>0$, denote by
$\operatorname{pack}_\delta(Y)$ the supremum of the cardinalities
of $\delta$-separated subsets of $Y$.

\begin{lem}[Packing estimates]\label{lem:packing-estimates}
Let $\Sigma$ be a compact $m$-dimensional Alexandrov space with
curvature bounded below by $1$.

\begin{enumerate}
\item
For every $v\in\Sigma$ and every $\nu>0$, there exists
$c=c(\Sigma,v,\nu)>0$ such that
\[
\operatorname{pack}_\delta B_\nu(v)
\ge c\,\delta^{-m}
\]
for every sufficiently small $\delta>0$.

\item
If $m\ge1$, there exists $C=C(m)$ such that for every $v\in\Sigma$
and every sufficiently small $\delta>0$,
\[
\operatorname{pack}_\delta
\left\{
w\in\Sigma:
\left||vw|-\frac{\pi}{2}\right|
\le\frac{\delta}{10}
\right\}
\le C\,\delta^{-(m-1)}.
\]
\end{enumerate}
\end{lem}

\begin{proof}
(1) Choose a maximal $\delta$-separated set
$T\subset B_{\nu/2}(v)$. By maximality, the balls
$B_\delta(w)$, $w\in T$, cover $B_{\nu/2}(v)$.
The absolute volume comparison theorem gives
\[
\mathcal H^m(B_\delta(w))
\le V_{1,m}(\delta)
\le C_0(m)\delta^m,
\]
where $V_{1,m}(\delta)$ denotes the volume of a ball of radius $\delta$ in the comparison space of dimension $m$ and constant curvature $1.$
Similarly, if $\# T$ denotes the cardinality of the set $T$, we obtain
\[
0<\mathcal H^m(B_{\nu/2}(v))
\le
\#T\,C_0(m)\delta^m,
\]
which proves (1).

(2) Observe that the statement is true in the case, where $\Sigma$ is the $m$-dimensional sphere. By \cite[Proposition 10.6.10]{MR1835418} there exists a noncontracting map $f\colon \Sigma \rightarrow S^m(1)$. From this the claim follows.
\end{proof}
\begin{proof}[Proof of Lemma \ref{pi/2}.]
Fix $v_p\in\Sigma_p$ and $\nu>0$. The case $n=1$ is immediate.
Let $c>0$ and $C=C(n-1)$ be the constants in
Lemma~\ref{lem:packing-estimates}. Choose $\delta>0$ so small that
$
c(2\delta)^{-(n-1)}>C\delta^{-(n-2)}.
$
Then there is a $2\delta$-separated family
$
w_1,\dots,w_N\in B_{\nu/2}(v_p)$, for some
$N>C\delta^{-(n-2)}.
$
After a small perturbation, assume that the $w_i$ are represented by
unique minimizing geodesics from $p$. For sufficiently small $r>0$,
let $q_i\in S_r(p)$ be the corresponding points.
The definition of angle implies, after decreasing $r$ if necessary,
that
$
\sphericalangle_\kappa q_i p q_j
>
\frac 32\delta$ for
$i\ne j.
$
By continuity of comparison angles, after shrinking
$\varepsilon>0$, we have
$
\sphericalangle_\kappa q_i x q_j>\delta$
for every $x\in B_\varepsilon(p)$ and every $i\ne j$.
Then the directions
$\uparrow_x^{q_i}$ remain $\delta$-separated for every
$x\in B_\varepsilon(p)$. By
Lemma~\ref{lem:packing-estimates} (2), for every $v_x\in\Sigma_x$
at most $C\delta^{-(n-2)}$ of them satisfy
$
\left|\frac{\pi}{2}
-\sphericalangle(v_x,\uparrow_x^{q_i})\right|
\le\frac{\delta}{10}.
$
Hence for some $i$ we have
$
\left|\frac{\pi}{2}
-\sphericalangle(v_x,\uparrow_x^{q_i})\right|
>\frac{\delta}{10}.
$
The claim follows with $\xi=\delta/10$.
\end{proof}

\subsection{Proof of the Proposition}

\begin{proof}[Proof of Proposition \ref{Prop:Existence of strictly concave functions}.]
Choose $r>0$, points $q_1,\dots,q_N\in S_r(p)$, and constants
$\varepsilon_0,\xi>0$ as in Lemma~\ref{pi/2}. Shrinking
$\varepsilon_0$ if necessary, choose $\lambda>0$ so that all functions
$\dist_{q_i}$ are $\lambda$-concave on $B_{\varepsilon_0}(p)$. Choose
$M>(2\lambda+1)N/\sin^2\xi$.

Consider
$\phi_M(s)=(s-r)-M(s-r)^2$.
Then $\phi_M(r)=0$, $\phi_M'(r)=1$, and
$0<\phi_M'(s)<2$, $\phi_M''(s)=-2M$ on
$[r-\frac1{4M},r+\frac1{4M}]$.
Set $\varepsilon=\min\{\varepsilon_0,\frac1{4M}\}$. Since
$q_i\in S_r(p)$, for every $x\in B_\varepsilon(p)$ we have
$|\dist_{q_i}(x)-r|\le |px|<\varepsilon$. Hence
$\dist_{q_i}(B_\varepsilon(p))\subset
[r-\frac1{4M},r+\frac1{4M}]$.
Set

$$
f=\frac1N\sum_{i=1}^N\phi_M\circ\dist_{q_i}.
$$

We claim that $f$ is $(-1)$-concave on $B_\varepsilon(p)$.

Let $\gamma$ be a unit-speed geodesic contained in $B_\varepsilon(p)$
and let $x=\gamma(t^*)$ be an interior point. Consider
$v \in T_xX^n$, where $v$ is the tangent vector in the direction of $\gamma$. By Lemma~\ref{pi/2}, there exist an index $i^*$ and
a direction $w$ representing a path from $x$ to $q_{i*}$ such that
$|\frac\pi2-\sphericalangle(v,w)|>\xi$.
If $\sphericalangle(v,w)<\pi/2-\xi$, we apply the first variation formula in the
forward direction of $\gamma$; if $\sphericalangle(v,w)>\pi/2+\xi$, we apply it
in the opposite direction. Thus

$$
\max\left\{
\left|(\dist_{q_{i^*}}\circ\gamma)'_-(t^*)\right|,
\left|(\dist_{q_{i^*}}\circ\gamma)'_+(t^*)\right|
\right\}>\sin\xi.
$$

Applying Lemma~\ref{deriv2} to
$u_i=\dist_{q_i}\circ\gamma$, we obtain
$(\phi_M\circ\dist_{q_{i^*}}\circ\gamma)''(t^*)
<2\lambda-M\sin^2\xi$
in the barrier sense, while for every $i\ne i^*$,
$(\phi_M\circ\dist_{q_i}\circ\gamma)''(t^*)<2\lambda$.
Averaging the corresponding upper barriers gives

$$
(f\circ\gamma)''(t^*)
<
\frac{2\lambda N-M\sin^2\xi}{N}
<-1.
$$

Hence $f$ is $(-1)$-concave on $B_\varepsilon(p)$.

It remains to control the differential at $p$. Given a prescribed
direction $v_p\in\Sigma_p$ and $\nu>0$, choose the points $q_i$ in
Lemma~\ref{pi/2} so that
$\sphericalangle(\uparrow_p^{q_i},v_p)<\nu$ for all $i$.
Since $|pq_i|=r$ and $\phi_M'(r)=1$, we have
$d_p(\phi_M\circ\dist_{q_i})=d_p\dist_{q_i}$, and therefore

$$
d_pf=\frac1N\sum_{i=1}^N d_p\dist_{q_i}.
$$

By the first variation formula,
$d_p\dist_{q_i}(v)=-|v|\cos\sphericalangle(v,\uparrow_p^{q_i})$.
Therefore, as $\nu\to0$, $d_pf$ converges uniformly on $\Sigma_p$ to
$v\mapsto-|v|\cos\sphericalangle(v,v_p)$.
Scaling $f$ by $|v_0|$ gives the required approximation to
$v\mapsto-\langle v_0,v\rangle$.
\end{proof}

\begin{cor}\label{maximum}
There exists a strictly concave function defined in a neighborhood of $p$
which attains a strict local maximum at $p$.
\end{cor}

\begin{proof}
Choose $\delta\in(0,\pi/2)$ and a finite $\delta$-net
$v_1,\dots,v_m\subset\Sigma_p$. For each $v_j$, apply Proposition \ref{Prop:Existence of strictly concave functions}
and, after adding a constant, obtain a $(-1)$-concave function $f_j$
with $f_j(p)=0$ and
$
|d_pf_j(v)+\cos\sphericalangle(v,v_j)|<\cos\delta
$
for all $v\in\Sigma_p$.

Set $F=\min_j f_j$. Since all $f_j$ are $(-1)$-concave, so is $F$, and
$d_pF(v)=\min_jd_pf_j(v)$. For every $v\in\Sigma_p$, choose $j$ with
$\sphericalangle(v,v_j)<\delta$. Then
$$
d_pF(v)\le d_pf_j(v)<-\cos\sphericalangle(v,v_j)+\cos\delta<0.
$$
Hence $p$ is a strict local maximum of $F$.
\end{proof}

\begin{lem}[Arbitrary small convex neighborhoods]\label{strict-max-convex-neighborhoods}
Let $X$ be an Alexandrov region and let $F$ be a continuous concave
function defined near $p\in X$. If $F$ has a strict local maximum at
$p$, then $p$ admits arbitrarily small convex neighborhoods.
\end{lem}

\begin{proof}
Choose a sufficiently small working radius $r$ such that
$\overline B_{2r}(p)$ lies in the domain of $F$ and
$F(x)<F(p)$ for $x\in B_{2r}(p)\setminus{p}$.
Set $m_r=\max_{S_r(p)}F$ and choose $m_r<c<F(p)$.
Let $\Omega$ be the component containing $p$ of $\lbrace {F>c\rbrace }$.
Then $\Omega\subset B_r(p)$.

For $x,y\in\Omega$, choose a minimizing geodesic between them inside
$B_{2r}(p)$. By concavity, it lies in $\lbrace {F>c\rbrace }$ and hence in the same
component $\Omega$. Thus $\Omega$ is convex. Since $r$ can be chosen
arbitrarily small, the claim follows.
\end{proof}

The immediate corollary is Proposition~\ref{main theorem}. \par 

If we have Alexandrov regions $(X_i,p_i)$ for $i \in \N$ and $(X,p)$, with working radii of uniform size around $p_i$, then the notion of pointed GH-convergence is defined for these balls. In this situation we use the shorthand notation $(X_i,p_i)\to(X,p)$. 

\begin{cor}[Non-collapsed stability]\label{noncollapsestab}
Let $(X_i,p_i)\to(X,p)$ in the pointed Gromov--Hausdorff topology,
where $\dim X_i=\dim X=n$, and assume a uniform lower bound for the
working radii at $p_i$.

Let $F$ be a strictly concave function constructed above with a strict
local maximum at $p$, and lift the finitely many points occurring in
its construction to $X_i$. Then the corresponding functions $F_i$ are
uniformly strictly concave for all sufficiently large $i$, and $p_i$
admits a convex neighborhood of uniform size.
\end{cor}

\begin{proof}
By lower semicontinuity of angles, the separation inequalities used in
Lemma~\ref{pi/2} persist in $X_i$. Since
$\dim\Sigma_x=n-1$ in both the limit and the approximating spaces, the
same equatorial packing estimate applies. Hence the construction gives
a uniform negative concavity bound for $F_i$.

Moreover, $F_i\to F$ uniformly on compact sets. Since $F$ has a strict
maximum at $p$, for some fixed $r,\eta>0$ and all large $i$,
$
F_i(p_i)>\max_{S_r(p_i)}F_i+\eta.
$
The functions $F_i$ have a uniform local Lipschitz bound, so a suitable
superlevel component contains a fixed ball around $p_i$ and is contained
in $B_r(p_i)$. By Lemma \ref{strict-max-convex-neighborhoods} it is convex.
\end{proof}

\iffalse

\fi

%

%	\textcolor{teal}{In view of Definition \ref{Def: Alexandrov region} one does not loose anything, if one considers the induced intrinsic metric $\hat{d}$ instead of the original metric $d$. Therefore we introduce the following convention: Henceforth metric balls $B_r(x)$ are always taken with respect to the induced intrinsic metric $\hat{d}$, i.e.
%	 $$ B_r(p) = \lbrace x \in A : \hat{d}(x,p)<r  \rbrace. $$}
%

\section{Proof of the Quantitative Convex Neighborhood Theorem}

\subsection{Comparison radius estimate}

We first obtain a quantitative version of the local-to-global principle.
Given a point $p \in X$ in an Alexandrov region we want to estimate the size of the largest ball $\overline{B}_r(p)$ around $p$ such that the comparison inequality \eqref{eq: comparison inequality} from \autoref{Def: Lower curvature bound} is satisfied for any quadruple of points.
\begin{defi}[Comparison radius]\label{Def: Comparison radius}
	Denote by $X$ an Alexandrov region. For any $p \in X$, the comparison radius at $p$ is defined to be 	
		     \[
    \operatorname{cr}_\kappa(p) := \sup \left\lbrace r > 0 : 
    \begin{aligned}
        &\text{for  any quadruple } p; a,b,c \in \overline{B}_r(p),\\
        &\text{inequality \eqref{eq: comparison inequality} with constant $\kappa$ holds,}
        \text{ and } \overline{B}_{2r}(p) \text{ is compact}
    \end{aligned}
   \right\rbrace.
    \]
	\end{defi}

\begin{prop}[Comparison radius estimate]\label{Prop: Comparison radius estimate}
There exists universal constant $c>0$ such that, if
$\overline B_r(p)$ is compact and $0<4r\le \varpi^\kappa=\pi/\sqrt\kappa$, then
${cr}_\kappa(p)\ge cr$.
\end{prop}

The proof uses the following radius-improvement step, adapted from the
globalization argument (Key lemma) in \cite{zbMATH07802912}.
\iffalse
\begin{lem}[Radius improvement]\label{Lem: Key Lemma}
There exists $K>1$ such that the following holds. Suppose
${cr}_\kappa(x)\ge\rho\le \varpi^\kappa=\pi/\sqrt\kappa$ for every $x\in B_\rho(q)$ and every two points of
$B_\rho(q)$ can be joined by a minimizing geodesic contained in the
ambient compact ball. Then ${cr}_\kappa(q)\ge K\rho$.
\end{lem}

\fi
%QQQQQQQQQQQQQQQQQ

\begin{lem}[Radius improvement]\label{Lem: Key Lemma}
There exist universal constants $A>1$ and $K>1$ with the following
property. Suppose that $\operatorname{cr}_\kappa(x)\ge \rho$ for all
$x\in B_{A\rho}(q)$, and that all minimizing geodesics involved in
$B_{A\rho}(q)$ exist in a fixed compact neighborhood. Then
\[
\operatorname{cr}_\kappa(q)\ge K\rho,
\]
provided the relevant scales are smaller than $\varpi^\kappa$.
\end{lem}

\begin{proof}
This is a localized form of the Key Lemma
\cite[8.35]{zbMATH07802912}. Using the equivalence between the local
four-point and hinge formulations of the CBB condition
\cite[8.30]{zbMATH07802912}, the assumption
$\operatorname{cr}_\kappa\ge\rho$ supplies the small-hinge comparison
required there. After changing the numerical constants, the Key Lemma
increases the comparison scale by a universal factor $K>1$.
Its proof is local, so it applies as long as all points and minimizing
geodesics involved remain in the prescribed compact neighborhood.
\end{proof}

%QQQQQQQQQQQQQQQQQQQQQ

\begin{rem}
The constant in Proposition~\ref{Prop: Comparison radius estimate} is
universal as long as the relevant scales are below $\operatorname{Diam}(\mathbb R_\kappa^2)$. This restriction is harmless
for our applications. Indeed, a lower curvature bound $\kappa$ may
always be replaced by any smaller one; in particular, we may assume
$\kappa\le0$, in which case
$\operatorname{Diam}(\mathbb R_\kappa^2)=\infty$.
\end{rem}
\begin{proof}[Proof of Proposition~\ref{Prop: Comparison radius estimate}]
After rescaling, assume $r=1$. By compactness of $\overline B_1(p)$
and the local curvature condition, there exists $\rho>0$ such that
${cr}_\kappa(x)\ge\rho$ for all $x\in\overline B_1(p)$.

Let $K,A>1$ be given by Lemma~\ref{Lem: Key Lemma}. 
Increasing $A$ if necessary, we may assume $A>3$.
Denote by $\rho>0$ a sufficiently small real number. Then, whenever $q\in B_{1-A\rho}(p)$, one has
$B_{A\rho}(q)\subset B_1(p)$. In particular, any two points of
$B_\rho(q)$ can be joined by a minimizing geodesic contained in
$B_{A\rho}(q)$. \par 

We intend to apply Lemma~\ref{Lem: Key Lemma}
successively. At scale $\rho$ the first application at every $q\in B_{1-As}(p)$ gives 
${cr}_\kappa(q) \ge K\rho$. After removing
from $B_1(p)$ a collar of width $A\rho$, in the next step we work at
scale $K\rho$ and remove an additional collar of width $AK\rho$, and
so on. After $m$ steps we obtain
$$
{cr}_\kappa(q)\ge K^m\rho
\text{ on }
B_{1-D_m}(p),
\qquad \text{ where }
D_m=A  \rho(1+K+\dots+K^{m-1}).
$$

Choose
$c_0=(K-1)/(2AK)$ and let $m$ be the first integer such that
$K^m\rho\ge c_0$. Then $K^m\rho<Kc_0$, and hence
$
D_m<\frac{AK^m\rho}{K-1}
<\frac{AKc_0}{K-1}=\frac12.
$
In particular $p\in B_{1-D_m}(p)$, and therefore
${cr}_\kappa(p)\ge K^m\rho\ge c_0$.
Rescaling back gives $c_\kappa(p)\ge c_0r$.
\end{proof}

\subsection{Local compactness estimates}

For a metric space $Z$, let $\beta(\varepsilon,Z)$ denote the minimal
cardinality of an $\varepsilon$-net in $Z$.

\begin{lem}[Small-scale estimates]\label{lem: Uniform bounds on nets on small scales}
There exist $c=c(n,\kappa)>0$ and $C=C(n,\kappa)>0$ such that the
following holds. If $X^n$ is an Alexandrov region,
$\overline B_r(p)$ is compact, and $0<r\le1$, then for
$0<\varepsilon\le\rho\le cr$,
$$
\beta(\varepsilon,B_\rho(p))
\le C\left(\frac{\rho}{\varepsilon}\right)^n,
\qquad
\mathcal H^n(B_\rho(p))
\le V_{\kappa,n}(\rho)\le C\rho^n.
$$
Here $V_{\kappa,n}(\rho)$ denotes the volume of a ball of radius $\rho$ in the comparison space of dimension $n$ and constant curvature $\kappa$.
\end{lem}

\begin{proof}
By Proposition~\ref{Prop: Comparison radius estimate}, after decreasing
$c$ if necessary, comparison holds throughout $B_{cr}(p)$.
Moreover, $\Sigma_p$ is an $(n-1)$-dimensional Alexandrov space, and
there exists a distance-nondecreasing map
$f:\Sigma_p\to S^{n-1}$; see \cite[Proposition 10.6.10]{MR1835418}.

Angle comparison implies that
$\log_p^{\kappa}:B_{cr}(p)\to\cone_\kappa(\Sigma_p)$ is distance-nondecreasing.
The map $f$ extends radially to a distance-nondecreasing map
$\cone_\kappa(\Sigma_p)\to\R_\kappa^n$. Hence $B_\rho(p)$ admits a
distance-nondecreasing map into the model ball of radius $\rho$.
The packing and volume estimates in the statement follow immediately.
\end{proof}

The next lemma turns these local estimates into the covering bound needed for Gromov precompactness.
\begin{lem}[Local-to-global net bound]\label{Local to global bound}
Let $(X,d)$ be a metric space, $p_0\in X$, and $L<L_0$. Suppose that
for some $R>0$ all balls $B_s(x)$ with
$x\in B_{L_0}(p_0)$ and $s\le R$ satisfy a uniform bound
$\beta(\varepsilon,B_s(x))\le f(\varepsilon,s)$.
Assume also that every point of $B_{L_0}(p_0)$ can be joined to $p_0$
by a minimizing geodesic. Then
$\beta(\varepsilon,B_L(p_0))$ is bounded in terms of
$f,R,$ and $L$ only.
\end{lem}

\begin{proof}
Put $M=f(R/2,R)$. We first construct a uniform $R/2$-net on larger
balls. We claim that for every $j\ge0$ with
$R+jR/2<L_0$,
$$
\beta\bigl(R/2,B_{R+jR/2}(p_0)\bigr)\le M^{j+1}.
$$
For $j=0$ this follows directly from the assumption.

Suppose the estimate holds for some $j$, and let
$x_1,\dots,x_N$ be an $R/2$-net of $B_{R+jR/2}(p_0)$ with
$N\le M^{j+1}$. If
$y\in B_{R+(j+1)R/2}(p_0)$, take a minimizing geodesic from $p_0$ to
$y$. A point of this geodesic at distance at most
$R+jR/2$ from $p_0$ lies within $R/2$ of some $x_i$. Hence
$y\in B_R(x_i)$ for some $i$, and therefore
$$
B_{R+(j+1)R/2}(p_0)\subset\bigcup_{i=1}^N B_R(x_i).
$$
Each ball $B_R(x_i)$ admits an $R/2$-net with at most $M$ points.
Consequently
$$
\beta\bigl(R/2,B_{R+(j+1)R/2}(p_0)\bigr)
\le M^{j+2}.
$$

Thus every ball $B_L(p_0)$, $L<L_0$, admits an $R/2$-net whose
cardinality is bounded in terms of $f,R$, and $L$ only. Finally,
cover $B_L(p_0)$ by the corresponding $R/2$-balls and replace each of
them by an $\varepsilon$-net of cardinality at most
$f(\varepsilon,R/2)$. This gives the required bound
$\beta(\varepsilon,B_L(p_0))\le F(\varepsilon,L)$.
\end{proof}

\subsection{Proof of the main theorem}

\begin{proof}[proof of \autoref{quantitative theorem}.]
Suppose the conclusion fails. 
By decreasing the lower curvature bound if necessary, we may assume
$\kappa<0$. After rescaling each ball to radius $1$, the same lower
curvature bound is preserved since $r_i\le1$.

After a normalization depending only on
$\kappa$ and rescaling, we obtain Alexandrov regions $X_i^n$ and points
$p_i$ such that $\overline B_1(p_i)$ is compact,
$\mathcal H^n(B_{1/2}(p_i))\ge v_0$, but no convex Alexandrov
neighborhood of $p_i$ contains $B_{1/i}(p_i)$.

For $x\in B_{19/20}(p_i)$, a fixed ball around $x$ is contained in
$B_1(p_i)$. Lemma~\ref{lem: Uniform bounds on nets on small scales}
therefore gives uniform local covering bounds. Minimizing geodesics from
$p_i$ to points of $B_{19/20}(p_i)$ exist inside $B_1(p_i)$ by
compactness and intrinsicness. Lemma~\ref{Local to global bound} now
gives uniform covering bounds for $B_{9/10}(p_i)$. Hence we can apply the classical Gromov's precompactness theorem, after passing
to a subsequence,
$
(\overline B_{9/10}(p_i),p_i)\longrightarrow(Y,p_\infty)
$
in the pointed Gromov--Hausdorff topology.

Set $B=B_{1/2}(p_\infty)$. We claim that, with its induced intrinsic
metric, $B$ is an Alexandrov region. Indeed, for $y\in B$ and
approximating points $y_i\to y$, Proposition~\ref{Prop: Comparison radius estimate}
gives a uniform comparison radius around $y_i$. Minimizing geodesics
between sufficiently close points remain in a fixed compact inner ball
and therefore subconverge to minimizing geodesics in $Y$. Thus the
restricted metric on $B$ is locally geodesic, so it agrees locally with
its intrinsic metric, and the comparison inequality passes to the
limit.

We next show that $\dim B=n$. First, $\dim B\le n$: otherwise an
$(n+1)$-strainer in $B$ would lift to $X_i$ for all large $i$.

For the opposite inequality, Lemma~\ref{lem: Uniform bounds on nets on small scales}
gives, uniformly in $i$,
$\mathcal H^n(B_\varepsilon(x))\le C\varepsilon^n$ for points in the
relevant inner ball. Hence
$$
v_0\le\mathcal H^n(B_{1/2}(p_i))
\le C\varepsilon^n \cdot \beta(\varepsilon,B_{1/2}(p_i)),
$$
and therefore
$\beta(\varepsilon,B_{1/2}(p_i))\ge c v_0\varepsilon^{-n}$.
Using the standard relation between covering and packing numbers and
passing to the limit, we obtain
$\operatorname{pack}_\varepsilon(B_{3/5}(p_\infty))
\ge c'\varepsilon^{-n}$, after changing constants.

If $\dim B=m<n$, the small-scale estimate applied in $B$ gives instead
$\operatorname{pack}_\varepsilon(B_{3/5}(p_\infty))
\le C'\varepsilon^{-m}$, a contradiction as $\varepsilon\to0$.
Thus $\dim B=n$.

We may therefore apply Corollary~\ref{noncollapsestab}. It gives
convex neighborhoods of $p_i$ containing a ball of radius
$\rho>0$ independent of $i$, contradicting the choice of the sequence.
\end{proof}

\iffalse

\fi
%\cite{CBBlec}
\bibliography{Bibliography}

\end{document}